\documentclass[oneside,english]{amsart}
\usepackage[T1]{fontenc}
\usepackage[latin9]{inputenc}
\usepackage{mathrsfs}
\usepackage{amstext}
\usepackage{amsthm}
\usepackage{amssymb}

\makeatletter
\numberwithin{equation}{section}
\numberwithin{figure}{section}
\numberwithin{table}{section}
\theoremstyle{plain}
\newtheorem{thm}{\protect\theoremname}[section]
\theoremstyle{definition}
\newtheorem{defn}[thm]{\protect\definitionname}
\theoremstyle{plain}
\newtheorem{lem}[thm]{\protect\lemmaname}
\theoremstyle{plain}
\newtheorem{cor}[thm]{\protect\corollaryname}

\makeatother

\usepackage{babel}
\providecommand{\corollaryname}{Corollary}
\providecommand{\definitionname}{Definition}
\providecommand{\lemmaname}{Lemma}
\providecommand{\theoremname}{Theorem}

\begin{document}
\title{further generalization of central sets theorem for partial semigroups
and vip systems }
\author{anik pramanick and md mursalim saikh}
\email{pramanick.anik@gmail.com}
\address{Department of Mathematics, University of Kalyani, Kalyani, Nadia-741235,
West Bengal, India.}
\email{mdmsaikh2016@gmail.com}
\address{Department of Mathematics, University of Kalyani, Kalyani, Nadia-741235,
West Bengal, India.}
\begin{abstract}
The Central Sets Theorem, a fundamental result in Ramsey theory, is
a joint extension of both Hindman's theorem and van der Waerden's
theorem. It was originally introduced by H. Furstenberg using methods
from topological dynamics. Later, using the algebraic structure of
the Stone-\v{C}ech compactification $\beta S$ of a semigroup $S$,
N. Hindman and V. Bergelson extended the theorem in 1990. H. Shi and
H. Yang established a topological dynamical characterization of central
sets in an arbitrary semigroup $\left(S,+\right)$, and showed it
to be equivalent to the usual algebraic characterization. D. De, N.
Hindman, and D. Strauss later proved a stronger version of the Central
Sets Theorem for semigroups in 2008. D. Phulara further genaralized
the result for commutative semigroups in 2015. Recently in his work,
Zhang generalized it further and proved the central sets theorem for
uncountably many central sets. We extend the theorem to arbitrary
adequate partial semigroups and VIP systems.
\end{abstract}

\maketitle

\section{Introduction}

The notion of the central subet of $\mathbb{N}$ was originally introduced
by Furstenberg via topological dynamics. The following is the Central
Sets Theorem proved by Furstenberg in 1981 \cite{key-3}.
\begin{thm}
\label{Original CST} ( The Original Central Sets Theorem) Let $l\in\mathbb{N}$
and for each $i\in\left\{ 1,2,\ldots,l\right\} $, let $\left(y_{i,n}\right)_{n=1}^{\infty}$
be a sequence in $\mathbb{Z}$. Let $C$ be a central subset of $\mathbb{N}.$
Then there exist sequences $\langle a_{n}\rangle_{n=1}^{\infty}$in
$\mathbb{N}$ and $\langle H_{n}\rangle_{n=1}^{\infty}$ in $\mathcal{P}_{f}\left(\mathbb{N}\right)$
such that
\begin{enumerate}
\item for all $n$, $\max H_{n}<\min H_{n+1}$ and
\item for all $F\in\mathcal{P}_{f}\left(\mathbb{N}\right)$ and all $i\in\left\{ 1,2,\ldots,l\right\} $,
$\sum_{n\in F}\left(a_{n}+\sum_{t\in H_{n}}y_{i,t}\right)\in C$.
\end{enumerate}
\end{thm}

Subsequently, V. Bergelson and N. Hindman extended this formulation
of the Central Sets Theorem to arbitrary semigroups, as presented
in \cite{key-4-1}. Later, in 2008, D. De, N. Hindman, and D. Strauss
established a stronger version of the Central Sets Theorem inÂ \cite{key-2}.

\begin{thm}
\label{De} Let $\left(S,+\right)$ be a commutative semigroup and
let $C$ be a central subset of $S$. Then there exist functions $\alpha:\mathcal{P}_{f}\left(^{\mathbb{N}}S\right)\to S\text{ and }H:\mathcal{P}_{f}\left(S^{\mathbb{N}}\right)\to\mathcal{P}_{f}\left(\mathbb{N}\right)$
such that
\begin{enumerate}
\item if $F,G\in\mathcal{P}_{f}\left(S^{\mathbb{N}}\right)$ and $F\subsetneq G$
then $\max H\left(F\right)<\min H\left(G\right)$ and
\item if $m\in\mathbb{N}$, $G_{1},G_{2},\ldots,G_{m}\in\mathcal{P}_{f}\left(S^{\mathbb{N}}\right)$,
$G_{1}\subsetneq G_{2}\subsetneq\ldots\subsetneq G_{m}$, and for
each $i\in\left\{ 1,2,\ldots,m\right\} $, $\langle y_{i,n}\rangle_{n=1}^{\infty}\in G_{i}$,
then 
\[
\sum_{i=1}^{m}\left(\alpha\left(G_{i}\right)+\sum_{t\in H\left(G_{i}\right)}y_{i,t}\right)\in C.
\]
\end{enumerate}
\end{thm}

In 2015, D. Phulara generalized the stronger version of the Central
Sets Theorem for commutative semigroups. The theorem is as follows:
\begin{thm}
\label{Pu} Let $\left(S,+\right)$ be a commutative semigroup, let
$r$ be an idempotent in $J\left(S\right)$, and let $\langle C_{n}\rangle_{n=1}^{\infty}$
be a sequence of members of $r$. Then there exist $\alpha:\mathcal{P}_{f}\left(\mathbb{^{\mathbb{N}}}S\right)\to S$
and $H:\mathcal{P}_{f}\left(\mathbb{^{\mathbb{N}}}S\right)\to\mathcal{P}_{f}\left(\mathbb{N}\right)$
such that 
\begin{enumerate}
\item if $F,G\in\mathcal{P}_{f}\left(\mathbb{^{\mathbb{N}}}S\right)$ and
$F\subsetneq G$ , then $\max H\left(F\right)<\min H\left(G\right)$
and
\item whenever $t\in\mathbb{N}$, $G_{1},G_{2},\ldots,G_{t}\in\mathcal{P}_{f}\left(^{\mathbb{N}}S\right)$,
$G_{1}\subsetneq G_{2}\subsetneq\ldots\subsetneq G_{t}$, $\mid G_{1}\mid=m$
and for each $i\in\left\{ 1,2,\ldots,t\right\} $, $f_{i}\in G_{i}$,
then
\end{enumerate}
\[
\sum_{i=1}^{t}\left(\alpha\left(G_{i}\right)+\sum_{s\in H\left(G_{i}\right)}f_{i}\left(s\right)\right)\in C_{m}.
\]

\end{thm}

In \cite{key-8}, Jillian McLeod establishes a version of Theorem
\ref{Original CST} valid for commutative adequate partial semigroups.
In \cite{key-10}, Kendra Pleasant and in \cite{key-4}, Arpita Ghosh,
independently but later, prove a version of Theorem \ref{De} for
commutative adequate partial semigroups. In \cite{key-9}, Dev Phulara
generalized the Central Sets Theorem for commutative semigroups. Later,
in 2021, N. Hindman and K. Pleasant proved the Central Sets Theorem
for adequate partial semigroups in \cite{key-6}. Here, we present
a generalization of {[}Theorem $3.6$, \cite{key-6}{]} by K. Pleasant
and N. Hindman, inspired by Zhang's approach \cite{key-13}, which
is the further generalization of D. Phulara's Central Sets Theorem.
Actually, Phulara generalized the Central Sets Theorem for commutative
semigroups using sequence of central sets. Whereas, we establish the
theorem for arbitrary adequate partial semigroups by using arbitrarily
many central sets.

Apart from that, we further generalize the Central Sets Theorem for
VIP systems in commutative adequate partial semigroups. We now give
a concise overview of VIP systems, which are polynomial-type configurations.
\begin{defn}
\label{def1.4} Let $\left(G,+\right)$ be an abelian group. A sequence
$\langle v_{\alpha}\rangle_{\alpha\in\mathcal{P}_{f}\left(\mathbb{N}\right)}$
in $G$ is called a VIP system if there exists some non-negative integer
$d$ (the least such $d$ is called the degree of the system) such
that for every pairwise disjoint $\alpha_{0},\alpha_{1},\ldots,\alpha_{d}\in\mathcal{P}_{f}\left(\mathbb{N}\right)$,
we have 
\[
\sum_{t=1}^{d+1}\left(-1\right)^{t}\sum_{\mathcal{B\in}\left[\left\{ \alpha_{0},\alpha_{1},\ldots,\alpha_{d}\right\} \right]^{t}}v_{\cup\mathcal{B}}=0.
\]
\end{defn}

In their paper \cite{key-5}, the authors generalize this notion to
partial semigroups. They defined VIP systems for partial semigroups
in the following way.
\begin{defn}
\label{def1.5} Let $\left(S,+\right)$ be a commutative partial semigroup.
Let $\langle v_{\alpha}\rangle_{\alpha\in\mathcal{P}_{f}\left(\mathbb{N}\right)}$
be a sequence in $S$, then $\langle v_{\alpha}\rangle_{\alpha\in\mathcal{P}_{f}\left(\mathbb{N}\right)}$
is called a VIP system if there exist some $d\in\mathbb{N}$ and a
function from $\mathcal{F}_{d}$ to $S\cup\left\{ 0\right\} $, written
$\gamma\to m_{\gamma}$, $\gamma\in$$\mathcal{F}_{d}$, such that
\end{defn}

\[
v_{\alpha}=\sum_{\gamma\subseteq\alpha,\gamma\in\mathcal{F}_{d}}m_{\gamma}\text{ for all }\alpha\in\mathcal{P}_{f}\left(\mathbb{N}\right).\text{ ( In particular, the sum is always defined.)}
\]
 The sequence $\langle m_{\gamma}\rangle_{\gamma\in\mathcal{F}_{d}}$
is said to generate the VIP systems $\langle v_{\alpha}\rangle_{\alpha\in\mathcal{P}_{f}\left(\mathbb{N}\right)}$.

Later, they proved the Central Sets Theorem for VIP systems in commutative
adequate partial semigroups.
\begin{defn}
The following are definitions and notations used throughout the paper:
\begin{enumerate}
\item Given a nonempty set $A$, $\mathcal{P}_{f}\left(A\right)$ denote
the collection of nonempty finite subsets of $A$ i.e.,
\[
\mathcal{P}_{f}\left(A\right)=\left\{ F:\emptyset\neq F\subseteq A\text{ \text{and \ensuremath{F} is finite}}\right\} 
\]
\item $\mathcal{J}_{m}=\left\{ t\in\mathbb{N}^{m}:t\left(1\right)<t\left(2\right)<...<t\left(m\right)\right\} $
\item For $d\in\mathbb{N}$, let $\mathcal{F}_{d}$ denote the family of
nonempty subsets of $\mathbb{N}$ having cardinality at most $d$
i.e., 
\[
\mathcal{F}_{d}=\left\{ A\subset\mathbb{N}:\vert A\vert\leq d\right\} 
\]
\item Let $\langle H_{n}\rangle_{n=1}^{\infty}$ be a sequence of finite
sets, then 
\[
FU\left(\langle H_{n}\rangle_{n=1}^{\infty}\right)=\left\{ \cup_{n\in F}H_{n}:F\in\mathcal{P}_{f}\left(\mathbb{N}\right)\right\} 
\]
\item Let $\langle H_{n}\rangle_{n=1}^{\infty}$, $H_{n}\in\mathcal{P}_{f}\left(\mathbb{N}\right)$
by $H_{n}<H_{n+1}\text{ we mean }\max H_{n}<\min H_{n+1}$
\item An \textit{IP} ring $\mathcal{F}^{\left(1\right)}$ is a set of the
form $\mathcal{F}^{\left(1\right)}=FU\left(\langle\alpha_{n}\rangle_{n=1}^{\infty}\right)$
, where $\langle\alpha_{n}\rangle_{n=1}^{\infty}$ is a sequence of
members of $\mathcal{P}_{f}\left(\mathbb{N}\right)$ such that $\max\alpha_{n}<\min\alpha_{n+1}$
for each $n$.
\end{enumerate}
\end{defn}

\section{ALGEBRAIC BACKGROUND}

In this section, we provide a brief overview of the Stone-\v{C}ech
compactification $\beta S$ of a semigroup $S$. Let $\left(S,\cdot\right)$
be a discrete semigroup, and let $\beta S$ denote the collection
of all ultrafilters on $S$, where each point of $S$ is identified
with its corresponding principal ultrafilter. The family $\left\{ \overline{A}:A\subseteq S\right\} $,
where $\overline{A}=\left\{ p\in\beta S:A\in p\right\} $ forms a
closed basis for the toplogy on $\beta S$. With this topology $\beta S$
becomes a compact Hausdorff space in which $S$ is dense, called the
Stone-\v{C}ech compactification of $S$. The operation of $S$ can
be extended to $\beta S$, turning $\left(\beta S,\cdot\right)$ into
a compact, right topological semigroup such that $S$ contained in
its topological center. Explicitly, for each $p\in\beta S$, the map
$\rho_{p}:\beta S\to\beta S$, defined by $\rho_{p}(q)=q\cdot p$
is continuous; and for every $x\in S$, the map $\lambda_{x}:\beta S\to\beta S$,
defined by $\lambda_{x}(q)=x\cdot q$, is also continuous. Given $p,q\in\beta S$
and $A\subseteq S$, it holds that $A\in p\cdot q$ if and only if
$\left\{ x\in S:x^{-1}A\in q\right\} \in p$, where $x^{-1}A=\left\{ y\in S:x\cdot y\in A\right\} $.
For an elementary introduction to semigroup $\left(\beta S,\cdot\right)$
and its combinatorial significance, one can see \cite{key-3}. An
element $p\in\beta S$ is called idempotent if $p\cdot p=p$. For
more detailed information about $\beta S$, readers are requested
to see \cite{key-7}. In this section, we discuss about partial semigroups.

Given a set $S$, and a natural binary operation, it is often convenient
to define the operation for only a subset of $S\times S$. Consider
for instance the set $\mathscr{P}_{f}\left(\mathbb{N}\right)$. If
we define $\varphi:\left(\mathscr{P}_{f}\left(\mathbb{N}\right),\coprod\right)\rightarrow\left(\mathbb{N},+\right)$
by $\varphi\left(F\right)=|F|$, then $\varphi$ is not a homomorphism.

However, if we let 

\[
A\coprod B=\begin{cases}
\begin{array}{c}
A\cup B\qquad\qquad\text{ if }A\cap B=\emptyset\\
\text{ undefined}\qquad\text{ if }A\cap B\neq\emptyset
\end{array}\end{cases}
\]

then $\varphi$ is a homomorphism on $\left(\mathscr{P}_{f}\left(\mathbb{N}\right),\coprod\right)$,
in the sense that $\varphi\left(A\coprod B\right)=\varphi\left(A\right)+\varphi\left(B\right)$
whenever $A\coprod B$ is defined.

Another case in which we may need to restrict the domain of the operation
occurs when the natural operation does not satisfy the closure property.
For example, given a sequence $\langle x_{n}\rangle_{n=1}^{\infty}$in
the semigroup $\left(S,.\right)$, let

\[
T=FP\left(\langle x_{n}\rangle_{n=1}^{\infty}\right)=\left\{ \prod_{n\in F}x_{n}:F\in\mathscr{P}_{f}\left(\mathbb{N}\right)\right\} ,
\]

where the products are taken in increasing order of indices. Then
$\left(x_{1}.x_{3}\right)\ldotp\left(x_{2}.x_{4}\right)$ are not
likely to be in $T$ unless $x_{2}$ and $x_{3}$ commute, and $\left(x_{1}.x_{3}\right)\ldotp\left(x_{2}.x_{4}\right)$
is not likely to be in $T$ at all. On the other hand, if we let $\left(\prod_{n\in F}x_{n}\right)*\left(\prod_{n\in G}x_{n}\right)$be 

\[
\begin{cases}
\begin{array}{c}
\prod_{n\in F\cup G}x_{n}\qquad\text{ if}\text{ }\max F<\min G\\
\text{undefined}\qquad\text{ if }\max F\geq\min G
\end{array}\end{cases}
\]

Then $T$ is closed under $*$.

$\left(\mathscr{P}_{f}\left(\mathbb{N}\right),\coprod\right)$ and
$\left(T,*\right)$ above, are examples of adequate partial semigroups,
which are defined next.
\begin{defn}
A partial semigroup is a pair $\left(S,*\right)$ where $*$ maps
a subset of $S\times S$ to $S$ and for all $a,b,c,\in S$, $\left(a*b\right)*c=a*\left(b*c\right)$
in the sense that if either side is defined, then so is the other
and they are equal.
\end{defn}

\begin{defn}
Let $\left(S,*\right)$ be a partial semigroup.
\begin{enumerate}
\item For $s\in S$, $\varphi\left(s\right)=\left\{ t\in S:s*t\text{ }\text{is\text{ }defined}\right\} $
\item For $H\in\mathcal{P}_{f}\left(S\right)$, $\sigma\left(H\right)=\bigcap_{s\in H}\varphi\left(s\right)$
\item $\sigma\left(\emptyset\right)=S$
\item For $s\in S$ and $A\subseteq S$, $s^{-1}A=\left\{ t\in\varphi\left(s\right):s*t\in A\right\} $
\item $\left(S,*\right)$ is \textit{adequate} if and only if $\sigma\left(H\right)\neq\emptyset$
for all $H\in\mathcal{P}_{f}\left(S\right)$.
\end{enumerate}
\end{defn}

\begin{lem}
Let $\left(S,*\right)$ be a partial semigroup, let $A\subseteq S$
and let $a,b,c\in S$. Then $c\in b^{-1}\left(a^{-1}A\right)\iff b\in\varphi\left(a\right)$
and $c\in\left(a*b\right)^{-1}A$. In particular, if $b\in\varphi\left(a\right)$,
then $b^{-1}\left(a^{-1}A\right)=\left(a*b\right)^{-1}A$.
\end{lem}

\begin{proof}
\cite{key-5}, Lemma $2.3$.
\end{proof}
Our primary interest lies in \textit{adequate partial semigroups},
as they naturally give rise to a distinguished subsemigroup of $\beta S$.
This subsemigroup inherits the structure of a \textit{compact right
topological semigroup}, whose formal definition provided below.
\begin{defn}
Let $\left(S,*\right)$ be a partial semigroup. Then 

$\delta S=\bigcap\overline{\varphi\left(x\right)}=\bigcap_{H\in\mathcal{P}_{f}\left(S\right)}\overline{\sigma\left[H\right]}$

Note that $\delta S\neq\emptyset$ when the partial semigroup $S$
is adequate and for $S$ being semigroup, $\delta S=\beta S$. 

For $\left(S,.\right)$ be a semigroup, $A\subseteq S$, $a\in S$,
and $p,q\in\beta S$ then $A\in a\cdot q\iff a^{-1}A\in q$

and

$A\in p\cdot q\iff\left\{ a\in S:a^{-1}A\in q\right\} \in p$

Now, we extend this notion for partial operation $*$. 

Let $\left(S,*\right)$ be an adequate partial semigroup. Then

$\left(a\right)$ For $a\in S$ and $q\in\overline{\varphi\left(a\right)}$,
$a*q=\left\{ A\subseteq S:a^{-1}A\in q\right\} $.

$\left(b\right)$ For $p\in\beta S$ and $q\in\delta S$, $p*q=\left\{ A\subseteq S:\left\{ a\in S:a^{-1}A\in q\right\} \in p\right\} $.
\end{defn}

\begin{lem}
Let $\left(S,*\right)$ be an adequate partial semigroup.
\begin{enumerate}
\item If $a\in S$ and $q\in\overline{\varphi\left(a\right)}$, then $a*q\in\beta S$.
\item If $p\in\beta S$ and $q\in\delta S$, then $p*q\in\beta S$.
\item Let $p\in\beta S$, $q\in\delta S,$ and $a\in S$. Then $\varphi\left(a\right)\in p*q$
if and only if $\varphi\left(a\right)\in p$.
\item If $p,q\in\delta S$, then $p*q\in\delta S$.
\end{enumerate}
\end{lem}

\begin{proof}
\cite{key-5}, Lemma $2.7$.
\end{proof}
\begin{lem}
Let $\left(S,*\right)$ be an adequate partial semigroup and let $q\in\delta S$.
Then the function $\rho_{q}:\beta S\to\beta S$ defined by $\rho_{q}\left(p\right)=p*q$
is continuous.
\end{lem}

\begin{proof}
\cite{key-5}, Lemma $2.8$.
\end{proof}
\begin{thm}
Let $\left(S,*\right)$ be an adequate partial semigroup. Then $\left(\delta S,*\right)$
is a compact Hausdorff right topological semigroup.
\end{thm}

\begin{proof}
\cite{key-5}, Theorem $2.10$.
\end{proof}
\begin{defn}
Let $p=p*p\in\delta S$ and let $A\in p$. Then $A^{*}=\left\{ x\in A:x^{-1}A\in p\right\} $.
\end{defn}

For an idempotent $p\in\delta S$ and $A\in p$, then $A^{*}\in p$.
\begin{lem}
\label{Lem1} Let $p=p*p\in\delta S$, let $A\in p$, let $x\in A^{*}$.
Then $x^{-1}\left(A^{*}\right)\in p$.
\end{lem}

\begin{proof}
\cite{key-5}, Lemma $2.12$.
\end{proof}
\begin{defn}
Let $\left(S,*\right)$ be a \textit{partial semigroup} and let $A\subseteq S$.
Then $A$ is \textit{syndetic} if and only if there is some $H\in\mathcal{P}_{f}\left(S\right)$
such that $\sigma\left(H\right)\subseteq\bigcup_{t\in H}t^{-1}A$.
\end{defn}

\begin{lem}
Let $\left(S,*\right)$ be an adequate partial semigroup and let $A\subseteq S$.
Then $A$ is syndetic if and only if there exists $H\in\mathcal{P}_{f}\left(S\right)$
such that $\delta S\subseteq\bigcup_{t\in H}\overline{t^{-1}A}$.
\end{lem}

\begin{proof}
\cite{key-5}, Lemma $2.14$.
\end{proof}
\begin{defn}
$K\left(\delta S\right)=\left\{ A:A\text{ is a minimal left ideal in }\delta S\right\} $.
\end{defn}

\begin{thm}
Let $\left(S,*\right)$ be an adequate partial semigroup and let $p\in\delta S$.
The following statements are equivalent:
\begin{enumerate}
\item $p\in K\left(\delta S\right)$.
\item for all $A\in p$, $\left\{ x\in S:x^{-1}A\in p\right\} $ is syndetic.
\item for all $q\in\delta S$, $p\in\delta S*q*p$.
\end{enumerate}
\end{thm}

\begin{proof}
\cite{key-5}, Theorem $2.15$.
\end{proof}
\begin{defn}
Let $\left(S,*\right)$ be an adequate partial semigroup and let $A\subseteq S$.
Then,
\begin{enumerate}
\item The set $A$ is \textit{piecewise} \textit{syndetic} in $S$ if and
only if $\overline{A}\bigcap K\left(\delta S\right)\neq\emptyset$.
\item The set $A$ is \textit{central} in $S$ if and only if there is some
idempotent $p$ in $K\left(\delta S\right)$ such that $A\in p$.
\item A set $A\subseteq S$ is a $J$-set if and only if for all $F\in\mathcal{P}_{f}\left(\mathcal{F}\right)$
and all $L\in\mathcal{P}_{f}\left(S\right)$, there exist $m\in N$,
$a\in S^{m+1}$, and $t\in\mathcal{J}_{m}$ such that for all $f\in F$,
\[
\text{\ensuremath{\left(\prod_{i=1}^{m}a\left(i\right)*f\left(t\left(i\right)\right)\right)}\ensuremath{\ensuremath{\ast\left(a\left(m+1\right)\right)}\ensuremath{\ensuremath{\in}A\ensuremath{\cap\sigma\left(L\right)}}}}.
\]
\item $J\left(S\right)=\left\{ p\in\delta S:\left(\forall A\in p\right)\left(A\text{ }is\text{ }a\text{ }J\text{-set}\right)\right\} $.
\end{enumerate}
\end{defn}

\begin{lem}
Let $\left(S,*\right)$ be an adequate partial semigroup and let $A\subseteq S$
be piecewise syndetic. There exists $H\in\mathcal{P}_{f}\left(S\right)$
such that for every finite nonempty set $T\subseteq\sigma\left(H\right)$,
there exists $x\in\sigma\left(T\right)$ such that $T*x\subseteq\bigcup_{t\in H}t^{-1}A$.
\end{lem}

\begin{proof}
\cite{key-5}, Lemma $2.17$.
\end{proof}
The following definition presents one of the essential notions in
this context: the adequate sequence for a partialÂ semigroup.
\begin{defn}
Let $\left(S,*\right)$ be an adequate partial semigroup and let $f$
be a sequence in $S$. Then $f$ is adequate if and only if
\begin{enumerate}
\item for each $H\in\mathcal{P}_{f}\left(\mathbb{N}\right)$, $\prod_{t\in H}f\left(t\right)$
is defined and
\item for each $F\in\mathcal{P}_{f}\left(S\right)$, there exists $m\in\mathbb{N}$
such that 
\[
FP\left(\left(f\left(t\right)\right)_{t=m}^{\infty}\right)\subseteq\sigma\left(F\right).
\]
\end{enumerate}
\end{defn}

\begin{defn}
Let $\left(S,*\right)$ be an adequate partial semigroup. Then 
\[
\mathcal{F}=\left\{ f:f\text{ is an adequate sequence in }S\right\} .
\]
\end{defn}

\section{FURTHER GENERALIZATION OF PHULARA'S CENTRAL SETS THEOREM FOR ADEQUATE
PARTIAL SEMIGROUPS}

In this paper, we prove that Theorem \ref{Pu} remains valid when
further generalized to arbitrary adequate partial semigroups. To establish
this result, we first require the following lemma.
\begin{lem}
\label{Lem2} Let $\left(S,*\right)$ be an adequate partial semigroup
and let $A$ be a $J$-set in $S$. Let $r\in\mathbb{N}$, let $F\in\mathcal{P}_{f}\left(\mathcal{F}\right)$,
and let $L\in\mathcal{P}_{f}\left(S\right)$. There exist $m\in\mathbb{N}$,
$a\in S^{m+1}$, and $t\in\mathcal{J}_{m}$ such that $t\left(1\right)>r$
and for all $f\in F$,
\end{lem}

$\prod_{i=1}^{m}a\left(i\right)*f\left(t\left(i\right)\right)*a\left(m+1\right)\in A\cap\sigma\left(L\right)$.
\begin{proof}
\cite{key-6}, Lemma $3.5$.
\end{proof}
We are now ready to present the main theorem of this section:
\begin{thm}
Let $\left(S,*\right)$ be an adequate partial semigroup and let $r$
be an idempotent in $J\left(S\right)$, and $R:\mathcal{P}_{f}(\mathcal{F})\to r$
is a function. Then there exist functions $m^{*}:\mathcal{P}_{f}\left(\mathcal{F}\right)\rightarrow\mathbb{N}$,
$\alpha\in\times_{F\in\mathcal{P}_{f}\left(\mathcal{F}\right)}S^{m^{*}\left(F\right)+1}$
and $\tau\in\times_{F\in\mathcal{P}_{f}\left(\mathcal{F}\right)}\mathcal{J}_{m^{*}\left(F\right)}$
such that 

(1) if $F,G\in\mathcal{P}_{f}(\mathcal{F})$ and $G\subsetneq F$,
then $\mathcal{\tau}(G)(m^{*}(G))<\mathcal{\tau}(F)(1)$ and 

(2) if $s\in\mathbb{N}$, $G_{1},G_{2},\ldots,G_{s}\in\mathcal{P}_{f}\left(\mathcal{F}\right)$,
$G_{1}\subsetneq G_{2}\subsetneq\ldots\subsetneq G_{s}$, and for
each $i\in\{1,2,\ldots,s\}$, $f_{i}\in G_{i}$, then

\[
\prod_{i=1}^{s}\left(\left(\prod_{j=1}^{m^{*}(G_{i})}\alpha(G_{i})(j)*f_{i}(\mathcal{\tau}(G_{i})(j)\right)*\alpha(G_{i})(m^{*}(G_{i})+1)\right)\in R\left(G_{1}\right).
\]
\end{thm}

\begin{proof}
Let $C^{*}=\left\{ x\in C:x^{-1}C\in r\right\} $. Then $C^{*}\in r$
and by Lemma \ref{Lem1}, for each $x\in C^{*}$, $x^{-1}C^{*}\in r$. 

We define $m^{*}\left(F\right)\in\mathbb{N}$, $\alpha\left(F\right)\in S^{m^{*}\left(F\right)+1}$,
$\mathcal{\tau}(F)\in\mathcal{J}_{m^{*}\left(F\right)}$ for $F\in\mathcal{P}_{f}\left(\mathcal{F}\right)$
by induction on $\vert F\vert$ satisfying the following induction
hypothesis:

(1) if $\emptyset\neq G\subsetneq F$, then $\mathcal{\tau}(G)(m^{*}(G))<\mathcal{\tau}(F)(1)$
and

(2) if $s\in\mathbb{N}$, $G_{1},G_{2},\ldots,G_{s}\in\mathcal{P}_{f}\left(\mathcal{F}\right)$,
$G_{1}\subsetneq G_{2}\subsetneq\ldots\subsetneq G_{s}=F$, and for
each $i\in\left\{ 1,2,\ldots,s\right\} $, $f_{i}\in G_{i}$, then

\[
\prod_{i=1}^{s}\left(\left(\prod_{j=1}^{m^{*}(G_{i})}\alpha(G_{i})(j)*f_{i}(\mathcal{\tau}(G_{i})(j)\right)*\alpha(G_{i})(m^{*}(G_{i})+1)\right)\in R\left(G_{1}\right)^{*}.
\]

Assume first that $F=\left\{ f\right\} $. Since $R\left(F\right)\in r,$
we have $R\left(F\right)^{*}\in r,$ $R\left(F\right)^{*}$ is a $J$-set,
pick $k\in\mathbb{N}$, $a\in S^{k+1}$, and $t\in\mathcal{J}_{k}$
such that $\prod_{j=1}^{k}a\left(j\right)*f\left(t\left(j\right)\right)*a\left(k+1\right)\in R\left(F\right)^{*}$.
Let $m^{*}=k$, $\alpha\left(F\right)=a$, and $\tau(F)=t$. Hypothesis
(1) holds vacuously. To verify hypothesis (2), let $s$$\in\mathbb{N}$
and assume that $G_{1},G_{2},\ldots,G_{s}\in\mathcal{P}_{f}\left(\mathcal{F}\right)$,
$G_{1}\subsetneq G_{2}\subsetneq\ldots\subsetneq G_{s}=F$, and for
each $i\in\left\{ 1,2,\ldots,s\right\} $ , $f_{i}\in G_{i}$. Then
$s=1$, $G_{1}=F=\left\{ f\right\} $, and $f_{1}=f$. So

\[
\begin{array}{c}
\prod_{i=1}^{s}\left(\left(\prod_{j=1}^{m^{*}(G_{i})}\alpha(G_{i})(j)*f_{i}(\mathcal{\tau}(G_{i})(j)\right)*\alpha(G_{i})(m^{*}(G_{i})+1)\right)=\\
\prod_{j=1}^{k}a\left(j\right)*f\left(t\left(j\right)\right)*a\left(k+1\right)\in R\left(F\right)^{*}
\end{array}.
\]

Now let $n>1$, let $F\in\mathcal{P}_{f}\left(\mathcal{F}\right)$
with $\vert F\vert=n$, and assume that for each proper (nonempty)
subset $G$ of $F$, $m^{*}\left(G\right)$, $\alpha\left(F\right)$,
and $\tau\left(G\right)$ have been defined satisfying hypothesis
(1) and (2). Let $K=\left\{ \tau\left(G\right)\left(m^{*}\left(G\right)+1\right):\emptyset\neq G\subsetneq F\right\} $
and let $d=\max K$.

For $m\in\left\{ 1,2,\ldots,n-1\right\} $, let 
\[
M_{m}=\left\{ \begin{array}{c}
\prod_{i=1}^{s}\left(\left(\prod_{j=1}^{m^{*}(G_{i})}\alpha(G_{i})(j)*f_{i}(\mathcal{\tau}(G_{i})(j)\right)*\alpha(G_{i})(m^{*}(G_{i})+1)\right):\\
s\in\mathbb{N},G_{1}\subsetneq G_{2}\subsetneq\ldots\subsetneq G_{s}\subsetneq F,\\
\text{ and for each }i\in\left\{ 1,2,\ldots,s\right\} ,f_{i}\in G_{i}
\end{array}\right\} .
\]
Then by hypothesis (2), $M_{m}\subseteq R\left(G\right)^{*}$. 

Let $A=R\left(F\right)^{*}\cap\bigcap_{m=1}^{n-1}\bigcap_{x\in M_{m}}x^{-1}R\left(G\right)^{*}$.
Then $A\in r$. Let $L=\cup_{m=1}^{n-1}M_{m}$. By Lemma \ref{Lem2},
pick $k\in\mathbb{N}$, $a\in S^{k+1}$, and $t\in\mathcal{J}_{k}$
such that $t\left(1\right)>d$ and for all $f\in F$, $\left(\prod_{j=1}^{k}a\left(j\right)*f\left(t\left(j\right)\right)\right)*a\left(k+1\right)\in A\cap\sigma\left(L\right)$.
Let $m^{*}\left(F\right)=k$, $\alpha\left(F\right)=a$, and $\tau\left(F\right)=t$.

To verify hypothesis (1), assume $\emptyset\neq G\subsetneq F$. Then
$\tau\left(G\right)\left(m^{*}\left(G\right)+1\right)\leq d<\tau\left(F\right)\left(1\right)$.
To verify hypothesis (2), let $s,\in\mathbb{N}$ and assume that $G_{1},G_{2},\ldots,G_{s}\in\mathcal{P}_{f}\left(\mathcal{F}\right)$,
$G_{1}\subsetneq G_{2}\subsetneq\ldots\subsetneq G_{s}=F$, and for
each $i\in\left\{ 1,2,\ldots,s\right\} $, $f_{i}\in G_{i}$.

Assume first that $s=1$. Then

\[
\begin{array}{c}
\prod_{i=1}^{s}\left(\left(\prod_{j=1}^{m^{*}(G_{i})}\alpha(G_{i})(j)*f_{i}\left(\tau\left(G_{i}\right)\right)\right)*\alpha(G_{i})(m^{*}(G_{i})+1)\right)=\\
\left(\prod_{j=1}^{k}a\left(j\right)*f_{1}\left(t\left(j\right)\right)\right)*a\left(k+1\right)\in A\subseteq R\left(F\right)^{*}\text{ and }F=G_{1}.
\end{array}
\]

Now assume that $s>1$. Let

$x=\prod_{i=1}^{s-1}\left(\left(\prod_{j=1}^{m^{*}(G_{i})}\alpha(G_{i})(j)*f_{i}\left(\tau\left(G_{i}\right)\right)\right)*\alpha(G_{i})(m^{*}(G_{i})+1)\right)$
and let 

$y=\left(\left(\prod_{j=1}^{m^{*}(G_{s})}\alpha(G_{s})(j)*f_{s}\left(\tau\left(G_{s}\right)\right)\right)*\alpha(G_{s})(m^{*}(G_{s})+1)\right)$.
Then $x\in M_{m}$ and 

$y=\left(\prod_{j=1}^{k}a\left(j\right)*f_{s}\left(t\left(j\right)\right)\right)*a\left(k+1\right)\in A\cap\sigma\left(L\right)\subseteq x^{-1}R\left(G_{1}\right)^{*}$.
Since $x\in L$ and $y\in\sigma\left(L\right)$, $x*y$ is defined
and 

$\prod_{i=1}^{s}\left(\left(\prod_{j=1}^{m^{*}(G_{i})}\alpha(G_{i})(j)*f_{i}\left(\tau\left(G_{i}\right)\right)\right)*\alpha(G_{i})(m^{*}(G_{i})+1)\right)=x*y\in R\left(G_{1}\right)^{*}\subseteq R\left(G_{1}\right)$.

The inductive construction is complete. Conclusions (1) and (2) of
the theorem follow from hypothesis (1) and (2) for $G_{s}=F$. Thus,
the proof is complete.
\end{proof}
\begin{cor}
Let $\left(S,*\right)$ be a commutative adequate partial semigroup
and let $r$ be an idempotent in $J\left(S\right)$ and $R:\mathcal{P}_{f}(\mathcal{F})\to r$
is a function. Then there exist functions $\gamma:\mathcal{P}_{f}\left(\mathcal{F}\right)\to S$
and $H:\mathcal{P}_{f}\left(\mathcal{F}\right)\to\mathcal{P}_{f}\left(\mathbb{N}\right)$
such that
\begin{enumerate}
\item if $F,G\in\mathcal{P}_{f}\left(\mathcal{F}\right)$ and $G\subsetneq F$,
then $\max H\left(G\right)<\min H\left(F\right)$ and
\item if $n\in\mathbb{N}$, $G_{1},G_{2},\ldots,G_{n}\in\mathcal{P}_{f}\left(\mathcal{F}\right)$,
$G_{1}\subsetneq G_{2}\subsetneq\ldots\subsetneq G_{n}$, and for
each $i\in\left\{ 1,2,\ldots,n\right\} $, $f_{i}\in G_{i}$, then
$\prod_{i=1}^{n}\left(\gamma\left(G_{i}\right)*\prod_{t\in H\left(G_{i}\right)}f_{i}\left(t\right)\right)\in R\left(G_{1}\right)$.
\end{enumerate}
\end{cor}

\begin{proof}
Let $m^{*}$, $\alpha$ and $\tau$ be as guaranteed by previous theorem.
For $F\in\mathcal{P}_{f}\left(\mathcal{F}\right)$, let $\gamma\left(F\right)=\prod_{j=1}^{m^{*}\left(F\right)+1}\alpha\left(F\right)\left(j\right)$
and let $H\left(F\right)=\left\{ \tau\left(F\right)\left(j\right):j\in\left\{ 1,2,\ldots,m^{*}\left(F\right)\right\} \right\} $.
\end{proof}

\section{FURTHER GENERALIZATION OF PHULARA'S VERSION OF THE CENTRAL SETS THEOREM
FOR VIP SYSTEMS IN PARTIAL SEMIGROUPS}

In this section, we focus on a special class of finite families of
VIP systems and aim to further generalize the Central Sets Theorem.
The definitions of VIP systems for commutative semigroups and adequate
partial semigroups have already been provided in the Introduction
(Definitions \ref{def1.4} and \ref{def1.5}).
\begin{defn}
\label{Def4.1} Let $\left(S,+\right)$ be commutative adequate partial
semigroup. A finite set $\left\{ \langle v_{\alpha}^{\left(i\right)}\rangle_{\alpha\in\mathcal{P}_{f}\left(\mathbb{N}\right)}:1\leq i\leq k\right\} $
of VIP systems is said to be \textit{adequate} if there exist $d,\text{ }t\in\mathbb{N}$,
a set $\left\{ \langle m_{\gamma}\rangle_{\gamma\in\mathcal{F}_{d}}:i\in\left\{ 1,2,\dots,k\right\} \right\} $,
a set of VIP systems 
\[
\left\{ \langle u_{\alpha}^{\left(i\right)}=\sum_{\gamma\subseteq\alpha,\gamma\in\mathcal{F}_{d}}n_{\gamma}^{\left(i\right)}\rangle_{\alpha\in\mathcal{P}_{f}\left(\mathbb{N}\right)}:i\in\left\{ 1,2,\ldots,t\right\} \right\} ,
\]
 and sets $E_{1},E_{2},\ldots,E_{k}\subseteq\left\{ 1,2,\ldots,t\right\} $
such that
\begin{enumerate}
\item For each $i\in\left\{ 1,2,\ldots,k\right\} $, $\langle m_{\gamma}\rangle_{\gamma\in\mathcal{F}_{d}}$
generates $\langle v_{\alpha}^{\left(i\right)}\rangle_{\alpha\in\mathcal{F}}$.
\item For every $H\in\mathcal{P}_{f}\left(S\right)$, there exists $m\in\mathbb{N}$
such that for every $l\in\mathbb{N}$ and pairwise distinct $\gamma_{1},\gamma_{2},\ldots,\gamma_{l}\in\mathcal{F}_{d}$
with each 
\[
\gamma_{i}\nsubseteq\left\{ 1,2,\ldots,m\right\} ,\text{ }\sum_{i=1}^{t}\sum_{j=1}^{l}n_{\gamma_{j}}^{\left(i\right)}\in\sigma\left(H\right)\cup\left\{ 0\right\} .
\]
( In particular, the sum is defined.)
\item $\text{\ensuremath{m_{\gamma}^{\left(i\right)}}=\ensuremath{\sum_{t\in E_{i}}n_{\gamma}^{\left(t\right)}}}$
for all $i\in\left\{ 1,2,\ldots,k\right\} $ and all $\gamma\in\mathcal{F}_{d}$.
\end{enumerate}
\end{defn}

\begin{defn}
Let $\left(S,+\right)$ be a commutative adequate partial semigroup
and let $\mathcal{A}\subseteq\mathcal{P}_{f}\left(S\right)$. $\mathcal{A}$
is said to be adequately partition regular family if for every finite
subset $H$ of $S$ and every $r\in\mathbb{N}$, there exists a finite
set $F\subseteq\sigma\left(H\right)$ having the property that if
$F=\cup_{i=1}^{r}C_{i}$ then for some $j\in\left\{ 1,2,\ldots,r\right\} $,
$C_{j}$ contains a member of $\mathcal{A}$. $\mathcal{A}$ is said
to be \textit{shift invariant} if for all $A\in\mathcal{A}$ and all
$x\in\sigma\left(A\right)$, $A+x=\left\{ a+x:a\in A\right\} \in\mathcal{A}$.
\end{defn}

Let us now mention some useful theorems from \cite{key-5} for the
proof of our main theorem.
\begin{thm}
\label{Theo4.3} Let $\left(S,+\right)$ be a commutative adequate
partial semigroup and let $k\in\mathbb{N}$. If $\left\{ \langle v_{\alpha}^{\left(i\right)}\rangle_{\alpha\in\mathcal{P}_{f}\left(\mathbb{N}\right)}:1\leq i\leq k\right\} $
is an adequate set of VIP systems in $S$, and $\beta\in\mathcal{P}_{f}\left(\mathbb{N}\right)$,
then the family 
\[
\mathcal{A}=\left\{ \begin{array}{c}
\left\{ a,a+v_{\alpha}^{\left(1\right)},a+v_{\alpha}^{\left(2\right)},\ldots,a+v_{\alpha}^{\left(k\right)}\right\} :\\
\alpha\in\mathcal{P}_{f}\left(\mathbb{N}\right),a\in\sigma\left(\left\{ v_{\alpha}^{\left(1\right)},v_{\alpha}^{\left(2\right)},\ldots,v_{\alpha}^{\left(k\right)}\right\} \right)\text{ and}\text{ }\alpha>\beta
\end{array}\right\} 
\]
 is adequately partition regular.
\end{thm}

\begin{proof}
\cite{key-5}, Theorem $3.7$.
\end{proof}
\begin{thm}
\label{Theo4.4} Let $\left(S,+\right)$ be a commutative adequate
partial semigroup and let $\mathcal{A}$ be a shift invariant, adequately
partition regular family of finite subsets of $S$. Let $E\subseteq S$
be piecewise syndetic. Then $E$ contains a member of $\mathcal{A}$.
\end{thm}

\begin{proof}
\cite{key-5}, Theorem $3.8$.
\end{proof}
\begin{thm}
\label{Theo4.5} Let $\left\{ \langle v_{\alpha}^{\left(i\right)}\rangle_{\alpha\in\mathcal{P}_{f}\left(\mathbb{N}\right)}:1\leq i\leq k\right\} $
be an adequate set of VIP systems and pick $d,\text{ }t\in\mathbb{N}$,
a set $\left\{ \langle m_{\gamma}^{\left(i\right)}\rangle_{\gamma\in\mathcal{F}_{d}}:1\leq i\leq k\right\} $,
a set of VIP systems
\end{thm}

$\left\{ \langle u_{\alpha}^{\left(i\right)}=\sum_{\gamma\subseteq\alpha,\gamma\in\mathcal{F}_{d}}n_{\gamma}^{\left(i\right)}\rangle_{\alpha\in\mathcal{P}_{f}\left(\mathbb{N}\right)}:1\leq i\leq t\right\} $,

and sets $E_{1},E_{2},\ldots,E_{k}\subseteq\left\{ 1,2,\ldots,t\right\} $
satisfying conditions (1), (2), and (3) of Definition \ref{Def4.1}.
Let $\alpha_{1},\alpha_{2},\ldots,\alpha_{s}\in\mathcal{P}_{f}\left(\mathbb{N}\right)$
with $\alpha_{1}<\alpha_{2}<\ldots<\alpha_{s}$. For $F\subseteq\left\{ 1,2,\ldots,s\right\} $,
$i\in\left\{ 1,2,\ldots,k\right\} $ and $\varphi\in\mathcal{F}_{d}$
with $\varphi>\alpha_{s}$, and $1\leq i\leq k$, let

$b_{\varphi}^{\left(i,F\right)}=\sum_{\psi\subseteq\cup_{j\in F}\alpha_{j},\mid\psi\mid\leq d-\mid\varphi\mid}m_{\varphi\cup\psi}^{\left(i\right)}$.

For $F\subseteq\left\{ 1,2,\ldots,s\right\} $, $i\in\left\{ 1,2,\ldots,k\right\} $,
and $\beta\in\mathcal{F}_{d}$ with $\beta>\alpha_{s}$, let

$q_{\beta}^{\left(i,F\right)}=\sum_{\varphi\subseteq\beta,\varphi\in\mathcal{F}_{d}}b_{\varphi}^{\left(i,F\right)}$.

Then $\left\{ \langle q_{\beta}^{\left(i,F\right)}\rangle_{\beta\in\mathcal{P}_{f}\left(\mathbb{N}\right),\beta>\alpha_{s}}:i\in\left\{ 1,2,\ldots,k\right\} ,\text{ }F\subseteq\left\{ 1,2,\ldots,s\right\} \right\} $
is an adequate set of VIP systems.
\begin{proof}
\cite{key-5}, Theorem $3.10$.
\end{proof}
Here is our main theorem of this section. 
\begin{thm}
\label{Mms} Let $\left(S,+\right)$ be commutative adequate partial
semigroup, let $p$ be an idempotent in $K\left(\delta S\right)$,
$R:\mathcal{P}_{f}(\mathbb{N})\to p$ is a function and let
\[
\left\{ \langle v_{\alpha}^{\left(i\right)}\rangle_{\alpha\in\mathcal{P}_{f}\left(\mathbb{N}\right)}:1\leq i\leq k\right\} 
\]
 be $k$-many adequate set of VIP systems. Then there exist sequences
$\langle a_{n}\rangle_{n=1}^{\infty}$ in $S$ and $\langle\alpha_{n}\rangle_{n=1}^{\infty}$
in $\mathcal{P}_{f}\left(\mathbb{N}\right)$ such that $\alpha_{n}<\alpha_{n+1}$
for every $n$ and for every $F\in\mathcal{P}_{f}\left(\mathbb{N}\right)$,
$\gamma=\cup_{t\in F}\alpha_{t}$ such that 
\end{thm}

\[
\left\{ \sum_{t\in F}a_{t}\right\} \cup\left\{ \sum_{t\in F}a_{t}+v_{\gamma}^{\left(i\right)}:1\leq i\leq k\right\} \subseteq R\left(F\right).
\]

\begin{proof}
Let $C_{n}\in p\in K\left(\delta S\right)$. We may assume that $C_{n+1}\subseteq C_{n}$
for each $n\in\mathbb{N}$ ( If not, consider $B_{n}=\cap_{i=1}^{n}C_{i}$,
so $B_{n+1}\subseteq B_{n}$). For each $n\in\mathbb{N}$, let 
\[
C_{n}^{*}=\left\{ x\in C_{n}:-x+C_{n}\in p\right\} .
\]
Then for each $x\in C_{n}^{*}$, $-x+C_{n}^{*}\in p$ by Lemma \ref{Lem1}.
Let 
\[
\mathcal{A}=\left\{ \left\{ a,a+v_{\alpha}^{\left(1\right)},a+v_{\alpha}^{\left(2\right)},\ldots,a+v_{\alpha}^{\left(k\right)}\right\} :\alpha\in\mathcal{P}_{f}\left(\mathbb{N}\right),a\in\sigma\left(\left\{ v_{\alpha}^{\left(1\right)},v_{\alpha}^{\left(2\right)},\ldots,v_{\alpha}^{\left(k\right)}\right\} \right)\right\} .
\]
Then by Theorem \ref{Theo4.3}, $\mathcal{A}$ is adequately partition
regular and $\mathcal{A}$ is trivially shift invariant. Since for
each $F\in\mathcal{P}_{f}\left(\mathbb{N}\right)$, $R\left(F\right)^{*}\in p$
and $p\in K\left(\delta S\right)$, $R\left(F\right)^{*}$ is piecewise
syndetic. So by Theorem \ref{Theo4.4}, for some $a_{1}\in S$ and
$\alpha_{1}\in\mathcal{P}_{f}\left(\mathbb{N}\right)$ such that 
\[
\left\{ a_{1},a_{1}+v_{\alpha_{1}}^{\left(1\right)},a_{1}+v_{\alpha_{1}}^{\left(2\right)},\ldots,a_{1}+v_{\alpha_{1}}^{\left(k\right)}\right\} \subseteq R\left(F\right)^{*}
\]
 for every $n\in\mathbb{N}$. Now we proceed the proof by induction
on $n$, let $n\in\mathbb{N}$ and assume that we have chosen $\langle a_{t}\rangle_{t=1}^{n}$
in $S$ and $\langle\alpha_{t}\rangle_{t=1}^{n}$ in $\mathcal{P}_{f}\left(\mathbb{N}\right)$
such that 

$\left(1\right)$ for $t\in\left\{ 1,2,\ldots,n-1\right\} $, if any,
$\alpha_{t}<\alpha_{t+1}$ and

$\left(2\right)$ for $\emptyset\neq F\subseteq\left\{ 1,2,\ldots,n\right\} $,
if $\gamma=\cup_{t\in F}\alpha_{t}$, then $\sum_{t\in F}a_{t}\in R\left(F\right)^{*}$
and for each $i\in\left\{ 1,2,\ldots,k\right\} $, $\sum_{t\in F}a_{t}+v_{\gamma}^{\left(i\right)}\in R\left(F\right)^{*}$. 

For each $\gamma\in FU\left(\langle\alpha_{t}\rangle_{t=1}^{n}\right)$
and each $i\in\left\{ 1,2,\ldots,k\right\} $, let 
\[
\langle q_{\beta}^{\left(i,\gamma\right)}\rangle_{\beta\in\mathcal{P}_{f}\left(\mathbb{N}\right)}=\langle v_{\gamma\cup\beta}^{\left(i\right)}-v_{\gamma}^{\left(i\right)}\rangle_{\beta\in\mathcal{P}_{f}\left(\mathbb{N}\right),\beta>\alpha_{n}}.
\]
 By Theorem \ref{Theo4.5}, the family,

\[
\begin{array}{c}
\left\{ \langle q_{\beta}^{\left(i,\gamma\right)}\rangle_{\beta\in\mathcal{P}_{f}\left(\mathbb{N}\right),\beta>\alpha_{n}}:1\leq i\leq k,\text{ }\gamma\in FU\left(\langle\alpha_{t}\rangle_{t=1}^{n}\right)\right\} \cup\\
\left\{ \langle v_{\beta}^{\left(i\right)}\rangle_{\beta\in\mathcal{P}_{f}\left(\mathbb{N}\right)}:1\leq i\leq k\right\} 
\end{array}
\]

is an adequate set of VIP systems. Let 
\[
\mathcal{B}=\left\{ \begin{array}{c}
\left\{ a\right\} \cup\left\{ a+v_{\alpha}^{\left(i\right)}:1\leq i\leq k\right\} \cup\\
\bigcup_{\gamma\in FU\left(\langle\alpha_{t}\rangle_{t=1}^{n}\right)}\left\{ a+q_{\alpha}^{\left(i,\gamma\right)}:1\leq i\leq k\right\} :\alpha\in\mathcal{P}_{f}\left(\mathbb{N}\right),\text{ }\alpha>\alpha_{n},\text{ and }\\
a\in\sigma\left(\left\{ v_{\alpha}^{\left(i\right)}:1\leq i\leq k\right\} \cup\left\{ q_{\alpha}^{\left(i,\gamma\right)}:1\leq i\leq k,\text{ }\gamma\in FU\left(\langle\alpha_{t}\rangle_{t=1}^{n}\right)\right\} \right)
\end{array}\right\} .
\]
Then by Theorem \ref{Theo4.3}, $\mathcal{B}$ is adequately partition
regular. Let 
\[
D=R\left(F\right)^{*}\cap\bigcap_{m=1}^{n}\left[\begin{array}{c}
\bigcap\left\{ -\sum_{t\in H}a_{t}+R\left(H\right)^{*}:\emptyset\neq H\subseteq\left\{ 1,2,\ldots,n\right\} ,\text{ }m=\min H\right\} \cap\\
\bigcap\left\{ \begin{array}{c}
-\left(\sum_{t\in H}a_{t}+v_{\gamma}^{\left(i\right)}\right)+R\left(H\right)^{*}:1\leq i\leq k,\,\\
\emptyset\neq H\subseteq\left\{ 1,2,\ldots,n\right\} ,\text{ and }\gamma=\cup_{t\in H}\alpha_{t}
\end{array}\right\} 
\end{array}\right].
\]

Where $\emptyset\neq F\subseteq\left\{ 1,2,\ldots,n+1\right\} $.
Then $D\in p$ and $D$ is piecewise syndetic. So, by Theorem \ref{Theo4.4},
for some $\alpha_{n+1}\in\mathcal{P}_{f}\left(\mathbb{N}\right)$
such that $\alpha_{n+1}>\alpha_{n}$ and some 
\[
a_{n+1}\in\sigma\left(\left\{ v_{\alpha_{n+1}}^{\left(i\right)}:1\leq i\leq k\right\} \cup\left\{ q_{\alpha_{n+1}}^{\left(i,\gamma\right)}:1\leq i\leq k\text{ and }\gamma\in FU\left(\langle\alpha_{t}\rangle_{t=1}^{n}\right)\right\} \right)
\]
 such that 
\[
\begin{array}{c}
\left\{ a_{n+1}\right\} \cup\left\{ a_{n+1}+v_{\alpha_{n+1}}^{\left(i\right)}:i\in\left\{ 1,2,\ldots,k\right\} \right\} \cup\\
\bigcup_{\gamma\in FU\left(\langle\alpha_{t}\rangle_{t=1}^{n}\right)}\left\{ a_{n+1}+q_{\alpha_{n+1}}^{\left(i,\gamma\right)}:i\in\left\{ 1,2,\ldots,k\right\} \right\} \subseteq D
\end{array}
\]
 Induction hypothesis $\left(1\right)$ trivially holds. To verify
hypothesis $\left(2\right)$, let $\emptyset\neq F\subseteq\left\{ 1,2,\ldots,n+1\right\} $
and let $\gamma=\cup_{t\in F}\alpha_{t}$ . If $n+1\notin F$, the
conclusion holds by assumption. If $F=\left\{ n+1\right\} $, then
we have 
\[
\left\{ a_{n+1}\right\} \cup\left\{ a_{n+1}+v_{\alpha_{n+1}}^{\left(i\right)}:i\in\left\{ 1,2,\ldots,k\right\} \right\} \subseteq D\subseteq R\left(F\right)^{*}.
\]
So, let us assume that $\left\{ n+1\right\} \subsetneq F$, let $H=F$
\textbackslash$\left\{ n+1\right\} $, and let $\mu=\cup_{t\in F}\alpha_{t}$.
Then $a_{n+1}\in D\subseteq-\sum_{t\in H}a_{t}+R\left(H\right)^{*}$.

Let $\gamma=\cup_{t\in H}\alpha_{t}$, and let $i\in$$\left\{ 1,2,\ldots,k\right\} $.
Then
\[
a_{n+1}+q_{\alpha_{n+1}}^{\left(i,\gamma\right)}\in D\subseteq-\left(\sum_{t\in H}a_{t}+v_{\gamma}^{\left(i\right)}\right)+R\left(H\right)^{*}
\]

and so $\left(\sum_{t\in H}a_{t}+v_{\gamma}^{\left(i\right)}\right)+\left(a_{n+1}+q_{\alpha_{n+1}}^{\left(i,\gamma\right)}\right)\in R\left(H\right)^{*}$.
That is, 
\[
\sum_{t\in F}a_{t}+v_{\mu}^{\left(i\right)}=\left(\sum_{t\in H}a_{t}+a_{n+1}\right)+\left(v_{\gamma}^{\left(i\right)}+q_{\alpha_{n+1}}^{\left(i,\gamma\right)}\right)\in R\left(H\right)^{*}\subseteq R\left(H\right)\subsetneq R\left(F\right).
\]

This completes the proof.
\end{proof}

\section{APPLICATION}

The application of Theorem \ref{Mms} is briefly discussed in this
section:

\begin{defn}
Let $l\in\mathbb{N}$, a set-monomial $\left(\text{over }\mathbb{N}^{l}\right)$
in the variable $X$ is an expression $m\left(X\right)=S_{1}\times S_{2}\times\ldots\times S_{l}$,
where for each $i\in\left\{ 1,2,\ldots,l\right\} $, $S_{i}$ is either
the symbol $X$ or a nonempty singleton subset of $\mathbb{N}$ (these
are called coordinate coefficients). The degree of the monomial is
the number of times the symbol $X$ appears in the list $S_{1},\ldots,S_{l}$.
For example, taking $l=3$, $m\left(X\right)=\left\{ 5\right\} \times X\times X$
is a set-monomial of degree 2, while $m\left(X\right)=X\times\left\{ 17\right\} \times\left\{ 2\right\} $
is a set-monomial of degree 1. A\emph{ set-polynomial} is an expression
of the form $p\left(X\right)=m_{1}\left(X\right)\cup m_{2}\left(X\right)\cup\ldots\cup m_{k}\left(X\right)$,
where $k\in\mathbb{N}$ and $m_{1}\left(X\right)\cup m_{2}\left(X\right)\cup\ldots\cup m_{k}\left(X\right)$
are set-monomials. The degree of a set-polynomial is the largest degree
of its set-monomial ``summands'', and its constant term consists
of the ``sum'' of those $m_{i}$ that are constant, i.e., of degree
zero. 
\end{defn}

\begin{lem}
\label{Lem 5.3} Let $l\in\mathbb{N}$ and let $\mathcal{P}$ be a
finite family of set polynomial over 
\[
\left(\mathcal{P}_{f}\left(\mathbb{N}^{l}\right),+\right)
\]
 whose constant terms are empty. Then there exists $q\in\mathbb{N}$
and an IP ring $\mathcal{F}^{\left(1\right)}=\left\{ \alpha\in\mathcal{P}_{f}\left(\mathbb{N}\right):\min\alpha>q\right\} $
such that $\left\{ \langle P\left(\alpha\right)\rangle_{\alpha\in\mathcal{F}^{\left(1\right)}}:P\left(X\right)\in\mathcal{P}\right\} $
is an adequate set of VIP systems.
\end{lem}

\begin{proof}
\cite{key-5}, Lemma $4.3$.
\end{proof}
\begin{thm}
Let $l\in\mathbb{N}$ and let $\mathcal{P}$ be a finite family of
set polynomial over $\left(\mathcal{P}_{f}\left(\mathbb{N}^{l}\right),+\right)$
whose constant terms are empty, and let $p$ be a minimal idempotent
in $\delta\left(\mathcal{P}_{f}\left(\mathbb{N}^{l}\right)\right)$,
and if $R:\mathcal{P}_{f}(\mathbb{N})\to p$ is a function. Then,
there exists sequences $\langle A_{n}\rangle_{n=1}^{\infty}$ in $\mathcal{P}_{f}\left(\mathbb{N}^{l}\right)$
and $\langle\alpha_{n}\rangle_{n=1}^{\infty}$ in $\mathcal{P}_{f}\left(\mathbb{N}\right)$
with $\alpha_{n}<\alpha_{n+1}$ for each $n$ and for every $F\in\mathcal{P}_{f}\left(\mathbb{N}\right)$,
we have $\left\{ A_{\gamma}\right\} \cup\left\{ A_{\gamma}+P\left(\gamma\right):P\in\mathcal{P}\right\} \subseteq R\left(F\right)$,
$\gamma=\cup_{t\in F}\alpha_{t}$ and $A_{\gamma}=\sum_{t\in F}A_{t}$.
\end{thm}

\begin{proof}
By Lemma \ref{Lem 5.3}, there is an \textit{IP} ring $\mathcal{F}^{\left(1\right)}$
such that $\left\{ \langle P\left(\alpha\right)\rangle_{\alpha\in\mathcal{F}^{\left(1\right)}}:P\left(X\right)\in\mathcal{P}\right\} $
is an adequate set of VIP systems. Thus Theorem \ref{Mms} applies.
\end{proof}
$\vspace{0.3in}$

\textbf{Acknowledgment:} The first author gratefully acknowledges
the support of the Grant CSIR-UGC NET Fellowship, File No. 09/106(0202)/2020-EMR-I,
and the second author gratefully acknowledges the support received
through the University Research Scholarship of the University of Kalyani
(ID: 1F7/URS/Mathematics/2023/S-502).

$\vspace{0.3in}$

\end{document}